\documentclass [12pt,english] {article}
\usepackage[utf8]{inputenc}
\usepackage{amsthm}
\usepackage{amsmath}
\usepackage{amsfonts}
\usepackage{amssymb}
\usepackage[all]{xy}
\usepackage{color}   
\usepackage{tikz} \usetikzlibrary{arrows}

\usepackage{listings}
\usepackage{setspace}
\usepackage[font=small,labelfont=bf]{caption}
\singlespace

\newtheorem{ejem}{Example}
\newtheorem{defi}[ejem]{Definition}
\newtheorem{teo}[ejem]{Theorem}
\newtheorem{prop}[ejem]{Proposition}
\newtheorem{lema}[ejem]{Lemma}

\newtheorem{coro}[ejem]{Corollary}

\numberwithin{ejem}{section}
\title{Recovering trees from the cohomology ring of their configuration spaces}
\author{Teresa I. Hoekstra Mendoza \footnote{Supported by CONACYT FORDECYT-PRONACES/39570/2020}}

\begin{document}
\maketitle
\begin{abstract}
    Given a tree $T$, the cohomology ring of its unordered configuration space $H^{\ast}(U\mathcal{D}^nT)$ is an exterior face algebra if $T$ is a binary core tree (if by removing the leaves from $T$ we obtain a binary tree), or if $n=4$. This means that every cup product is determined by a simplicial complex $K_nT$. In this paper we show how to recover the tree $T$ from the simplicial complex $K_nT$ when $n=4.$
\end{abstract}


\section{Introduction}
For a finite graph $G$ and a positive integer $n$, the discretized unlabelled configuration space on $n$ points of $G$ is defined as $$U\mathcal{D}^nG=\{\{x_1, \dots,x_n \}: x_i \in V(G)\cup E(G), x_i\cap x_j = \emptyset \mbox{ if } i \neq j\}.$$
Configuration spaces of graphs have been widely studied, in particular, their cohomology ring is well known when $G$ is a tree.
Using discrete Morse theory techniques, D. Farley gave in \cite{F1} an efficient description of the additive structure of the cohomology ring of $U\mathcal{D}^nG$ when $G$ is a tree $T$. Later, and in order to get to the multiplicative structure, the Morse theoretic methods were replaced in \cite{F2} by the use of a Salvetti complex $\mathcal{S}$ obtained by identifying opposite faces of cells in $U\mathcal{D}^nT$.
Being a union of tori, $\mathcal{S}$ has a well understood cohomology ring.

 Given a tree $T$ we can construct another tree $F(T)$, where the vertex set of $F(T)$ is $V(F(T))= \{ x \in V(T): d(x)>2\}$ and two vertices are adjacent in $F(T)$ if the unique path joining them in $T$ does not contain any other vertex of $F(T)$. We say that a tree $T$ is a binary core tree if $F(T)$ is a binary tree.
In \cite{GH}, the following theorem was proven:
\begin{teo}{\cite{GH}}
Given an integer $n\geq 4$, the cohomology ring $H^{\ast} (U\mathcal{D}^nT)$ is an exterior face ring if either $n=4$, or $T$ is a binary core tree.
\end{teo}
 This means that when $T$ is a binary core tree or $n=4$, all products in $H^{\ast}(UD^nT)$ are given by a simplicial complex $K_nT$, where the vertices of $K_nT$ are the basis elements of dimension 1, and a set of $k$ vertices forms a simplex in $K_nT$ if and only if the cup product of the corresponding $k$ elements is non zero. 
The simplicial complex $K_nT$ is called the $n$-interaction complex of $T$.

\

The 1-skeleton of the space $U\mathcal{D}^nT$, coincides precisely with the token graph $F_n(T)$ of $T$ on $n$ tokens.
In \cite{R}, Fabila-Monroy and Trujillo-Negrete proved that if a graph does not have induced cycles of length four, or induced diamonds (a graph isomorphic to $K_4\setminus \{e\}$ the complete graph on four vertices without an edge), then the graph $G$ is uniquely reconstructible from $F_k(G)$  
and gave an algorithm to do so in polynomial time. One then might ask if there is a topological version of this, this is, if we can recover the tree $T$ from a topological property of $U\mathcal{D}^nT$.  We are going to prove that we can recover the tree $T$ from the $n$-interaction complex $K_nT$ in the case when $n=4$, for any tree $T$.

\

We shall assume that the tree $T$ is embedded in the plane and has as root a vertex of degree one $\star$. The edges inciding in a vertex $x$ are enumerated by this embedding and we fix the edge that lies on the unique $x\star$ path to be the $0$th edge.  We shall say that a vertex $y$ lies on $x$-direction $i$ for $1 \leq i \leq d(x)-1$ if $x$ belongs to the unique path joining $y$ and the root vertex, and if this path contains the $i$th edge inciding in $x$. This also implies that $x$ lies on $y$-direction $0$. If $x$ does not belong to the $y\star$ path and $y$ does not belong to the $x\star $ path we shall also say that $x$ lies on $y$-direction $0$ and $y$ lies on $x$-direction $0$, in which case we say that the vertices $x$ and $y$ are not stacked. 
Given two vertices of degree at least three $x$ and $y$, we shall say that $x<y$ if they are stacked and $x$ lies on $y$-direction zero, or if they are not stacked and there exists $z$ such that $x$ lies on $z$-direction $i$ and $ y$ lies on $z$-direction $j$ with $i<j$.
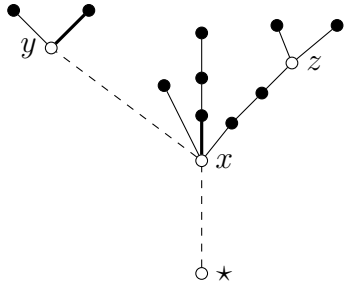
\begin{figure}[h!]
    \centering
  \begin{tikzpicture}
\node[circle, draw, scale=.4] (1xz) at (5,5){};
\node (0) at (5.3,5){$\star$};
\node[circle, draw, scale=.4] (2xz) at (5,6.5){};
\node (x) at (5.3,6.5){$x$};
\node[circle, draw, scale=.4] (3xz) at (3,8){};
\node (y) at (2.7,8){$y$};
\node[circle, draw, scale=.4, fill=black] (4xz) at (2.5, 8.5){};
\node[circle, draw, scale=.4,fill=black] (5xz) at (3.5,8.5){};
\node[circle, draw, scale=.4, fill=black] (6xz) at (4.5,7.5){};
\node[circle, draw, scale=.4,fill=black] (7xz) at (5,7.1){};
\node[circle, draw, scale=.4, fill=black] (8xz) at (5,7.6){};
\node[circle, draw, scale=.4, fill=black] (9xz) at (5,8.2){};
\node[circle, draw, scale=.4, fill=black] (10xz) at (5.4, 7){};
\node[circle, draw, scale=.4, fill=black] (11xz) at (5.8,7.4){};
\node[circle, draw, scale=.4] (12xz) at (6.2,7.8){};
\node (z) at (6.5,7.8){$z$};
\node[circle, draw, scale=.4, fill=black] (n) at (6.8,8.3){};
\node[circle, draw, scale=.4, fill=black] (m) at (6,8.3){};
\draw (n)--(12xz);
\draw (m)--(12xz);

\draw (11xz)--(12xz);
\draw (11xz)--(10xz);
\draw (2xz)--(10xz);
\draw[dashed] (1xz) -- (2xz);
\draw[dashed] (2xz)--(3xz);
\draw[very thick] (3xz)--(5xz);
\draw (4xz)--(3xz);
\draw (2xz)--(6xz);
\draw[very thick] (2xz)--(7xz);
\draw (7xz)--(8xz);
\draw (8xz)--(9xz);
\end{tikzpicture}
    \caption{The tree $T$ containing the vertices $x,y$ and $z$.}
    \label{stack}
\end{figure}

\begin{ejem}
Consider the tree shown in Figure \ref{stack}, where the dashed edges represent paths of any length greater than two. Then the vertices $x$ and $y$ are stacked, the vertices $z$ and $x$ are also stacked but the vertices $y$ and $z$ are not stacked.
\end{ejem}

We are going to define the interaction complex $K_nT$ only for the case when $n=4$, in which case it is a graph.
For the the definition of the interaction complex in the general case we refer the reader to \cite{GH}.

\section{The graph $K_4T$}

\begin{defi}
The graph $K_4T$ is defined as follows:
\begin{itemize}
    \item The vertices are 4-tuples $(k,x,p,q)$ such that $x$ is a vertex of $T$ with $d(x) \geq 3$, $k$ is a non-negative integer and $p$ and $q$ are integer vectors having non-negative entries such that $p$ has at least one positive entry, the sum of their lengths $l(p)+l(q)$ is $d(x)-1$ and the sum of their entries is $3-k$.
    \item Let $v=(k_1, x_1, p_1, q_1)$ and $w=(k_2, x_2, p_2, q_2)$ be two vertices, then
    \begin{enumerate}
        \item if the vertices $x_1$ and $x_2$ are not stacked, then $(v,w)\in E(K_4T)$ if $k_1 +k_2 \geq 4$
        \item if $x_2$ lies on $x_1$-direction $i$ with $l(p_1)\geq i$, then $(v,w)\in E(K_4T)$ if $p_{1,i}> 4-k_2$ or $p_{1,i}+k_2=4$ and there exists a $j \neq i$ such that $p_{1,j} \neq 0$
        \item if $x_2$ lies on $x_1$-direction $i$ with $l(p_1)< i$, then $(v,w)\in E(K_4T)$ if $q_{1,i-l(p_1)}> 4-k_2$ 
    \end{enumerate}
\end{itemize}
\end{defi}

In particular, the graph $K_4T$ has a lot of isolated vertices, as we can see in the following lemma.
\begin{lema}
A vertex $v=(k,x,p,q)$ is an isolated vertex if one of the following conditions hold:
\begin{itemize}
    \item $k=1$
    \item the sum of the entries of $p$ is two.
\end{itemize}
\end{lema}
\begin{proof}
   Assume first that $k=1$. Notice that since $l \in \{0,1,2\}$, if a vertex $w=(l,y,r,s)$ is such that $x$ and $y$ are not stacked, then both $k$ and $l$ must be two for $v$  and $w$ to be adjacent. If $x$ and $y$ are  stacked, every entry of $p=(p_1, \dots, p_{l(p)})$ is at most two, thus $2 \geq p_i$ and $ 4-l \geq 2$. Finally notice that every entry of $q=(q_1, \dots, q_{l(q)})$ is at most one so  $1 \geq q_{i-l(p_1)}$ and $ 4-k_2 \geq 2.$ This proves that every vertex with $k=1$ is isolated.
 Assume now that the sum of the entries of $p$ is two, and consider again $w=(l,y,r,s)$. This means that $k=0$ and $q$ has one unique entry with value one. Then again  $2 \geq p_i$, $ 4-l \geq 2$, $1 \geq q_{i-l(p_1)}$ and $ 4-k_2 \geq 2$  thus $v$ is an isolated vertex.
 \end{proof}
 For our purposes, the isolated vertices are not relevant, so we shall ignore them.
\begin{prop}\label{num}
Given a vertex $x$ of $T$, fix $i\in \{ 1, \dots, d(x)-1\}$ and let $$\Lambda_i=\{(k,x,p,q)\in K_4T: p_i\geq 2 \mbox{ if } i\leq l(p) \mbox{ or } q_{i-l(p)}\geq 2 \mbox{ if } i > l(p)\} .$$
Then the set $\Lambda_i$ has cardinality $\frac{1}{2}(d(x)-1)(d(x)-2)$, and this is also the amount of vertices having $k=2.$
\end{prop}
\begin{proof}
Assume first that $k=0$.
Since the sum of the entries of $p=(p_1, \dots, p_{l(p)})$ and $q$ must be at most $3-k=3$, and one entry is fixed to be 2, there can only be one additional non zero entry. Recall also that $p$ must have at least one zero entry, and since we are considering only vertices in which the sum of the entries of $p$ is not two, $p$ must have an additional non zero entry or $p_i =3$. Then for every possible length of $p$, we have $l(p)$ options, one for each additional 1 as an entry. Since $l(p)$ varies between $1$ and $d(x) -2$ we have $\sum_{i=1}^{d(x)-2} i= (d(x)-1)(d(x)-2).$ The same happens if $k=2$, since the remaining 1 must be an entry of $p$.
\end{proof}
\begin{figure}[h!]
\centering
\begin{tikzpicture}[every node/.style={circle, draw, scale=.6}, scale=1.0,
rotate = 180, xscale = -1]

\node (1) at ( 2, 2.58) {};
\node (2) at ( 1, 3.5) {};
\node (3) at ( 3, 3.5) {};
\node (13) at ( 2, 4.57) {};
\node (14) at ( 2, 1.46) {};
\node (15) at ( 0, 3.5) {};
\node (16) at ( 2, 5.86) {};
\node (17) at ( 4, 3.5) {};
\node (18) at ( 5.5, 2.73) {};
\node (19) at ( 6.0, 3.68) {};
\node (20) at ( 6.65, 2.69) {};
\node (21) at ( 5.59, 4.68) {};
\node (22) at ( 6.66, 4.67) {};
\node (23) at ( 2.86, 2) {};
\node (24) at ( 1.32, 1.69) {};
\node (25) at ( 0.6, 2.73) {};
\node (26) at ( 0.6, 4.47) {};
\node (27) at ( 1.44, 5) {};
\node (28) at ( 2.5, 5.0) {};
\node (29) at ( 4, 2.82) {};
\node (30) at ( 3.7, 4.5) {};

\draw (13) -- (2);
\draw (3) -- (13);
\draw (1) -- (3);
\draw (2) -- (1);
\draw (13) -- (1);
\draw (3) -- (2);
\draw (14) -- (1);
\draw (17) -- (3);
\draw (16) -- (13);
\draw (15) -- (2);
\draw (19) -- (18);
\draw (21) -- (19);
\draw (22) -- (19);
\draw (20) -- (19);
\draw (1) -- (24);
\draw (23) -- (1);
\draw (29) -- (3);
\draw (30) -- (3);
\draw (28) -- (13);
\draw (27) -- (13);
\draw (2) -- (25);
\draw (26) -- (2);

\node (1x) at ( 9.44, 2.08) {};
\node (2x) at ( 10.78, 1.72) {};
\node (3x) at ( 11.88, 2.44) {};
\node (4x) at ( 11.89, 3.43) {};
\node (5x) at ( 11.18, 4.27) {};
\node (6x) at ( 10.12, 4.22) {};
\node (7x) at ( 9.35, 3.2) {};
\node (8x) at ( 10.88, 0.63) {};
\node (9x) at ( 12.5, 1.8) {};
\node (1x0) at ( 12.99, 3.67) {};
\node (1x1) at ( 11.71, 5.0) {};
\node (1x2) at ( 9.9, 5.32) {};
\node (1x3) at ( 8.13, 3.12) {};
\node (1x4) at ( 8.67, 1.31) {};

\draw (1x) -- (1x4);
\draw (2x) -- (8x);
\draw (9x) -- (3x);
\draw (1x0) -- (4x);
\draw (1x1) -- (5x);
\draw (1x2) -- (6x);
\draw (7x) -- (1x3);
\draw (7x) -- (1x);
\draw (6x) -- (7x);
\draw (5x) -- (6x);
\draw (4x) -- (5x);
\draw (3x) -- (4x);
\draw (2x) -- (3x);
\draw (1x) -- (2x);
\draw (5x) -- (1x);
\draw (2x) -- (5x);
\draw (6x) -- (2x);
\draw (3x) -- (6x);
\draw (7x) -- (3x);
\draw (4x) -- (7x);
\draw (1x) -- (4x);
\draw (3x) -- (1x);
\draw (5x) -- (3x);
\draw (7x) -- (5x);
\draw (2x) -- (7x);
\draw (4x) -- (2x);
\draw (6x) -- (4x);
\draw (1x) -- (6x);

\end{tikzpicture}
\caption{The graph $K_4T$, the tree $T=K_{1,4}$ and an example of the graph $KP_7$.}
\label{n1}
\end{figure}
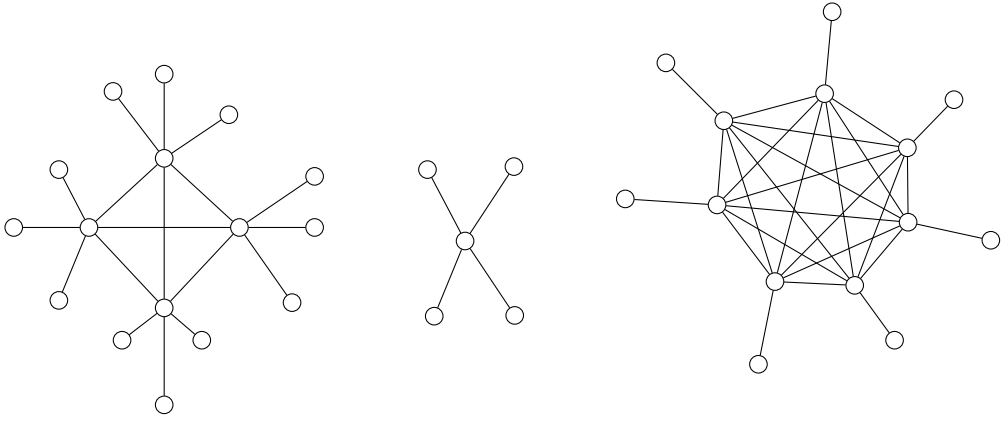

\begin{ejem}
When $T=K_{1,m}$, the graph $K_4T$ is the graph obtained from the complete graph $K_m$ by adding to each vertex, $\frac{1}{2}(m-1)(m-2)$ new neighbours as in Figure \ref{n1}.
\end{ejem}
The following proposition follows from the definition of $K_4T$.
\begin{prop}\label{hojas}
Let $v=(k,x,p,q)$ be a vertex in $K_4T$. Then the degree of $v$ is the amount of vertices in $F(T)$ which lie on direction $i$, where
\begin{itemize}
    \item  $i=0$ if $k=2$,
    \item  $i$ is such that the $i$th entry of $p$ is at least two if $ i \leq l(p)$,
    \item $i$ is such that the $(i+l(p))$th entry of $q$ is two if $i >l(p)$.
\end{itemize}
\end{prop}


\begin{coro}\label{coro}
Given a binary tree $T$, the number of leaves in $K_4T$ is the number of leaves in $F(T)$. 
\end{coro}
\begin{proof}
Let $w$ be a leaf in $F(T)$ different from the root, and let $v$ be its unique neighbour.
Assume $w$ lies on $v$-direction $i$. Then the vertex $u$ is a leaf in $K_4T$ and has as unique neighbour the vertex $(2,v,1,0)$, where 
$u=(0,w,p,q)$ is such that
$u=\left\{ \begin{array}{cc}
  (0,w,3,0)   & \mbox{ if } i=1  \\
    (0,w,1,2)  & \mbox{ if } i=2. \\
\end{array}\right.$
If $w$ is the root vertex with unique neighbour $v$, then in $K_4T$ the vertex $(2,v,1,0)$ is a leaf and its unique neighbour is $(0,w,3,0).$
\end{proof}

To simplify notation, we shall often write $0$ to denote the vector having every entry zero. Let $L(G)$ denote the set of leaves of a graph $G$.
\begin{prop}\label{vec}
If two vertices $u$ and $v$ in $K_4T$ have the same neighbourhood then $u=(k,x,p,q)$ and $v=(l,x,r,s)$.
\end{prop}
\begin{proof}
Assume $u=(k,x,p,q)$ and $v=(l,y,r,s)$ have the same (non empty) neighbourhood for $x \neq y$. Notice first that $u$ and $v$ are not adjacent since $K_4T$ has no loops. Notice also that $u$ and $v$ can not be both leaves, since that would mean that $u=(2,x,p,0)$ and $v=(2,y,r,0)$ thus $(u,v) \in E(K_4T)$, a contradiction.

 Assume first that $x$ and $ y$ are not stacked. This means that either $ k$ or $l$ is different that 2. Assume without loss of generality that $k\neq 2$, thus $k=0$. Since $N(u) \neq \emptyset$, there exists $w \in N(u)$, $w=(2, z, p',q')$ and $z$ lies on $x$-direction $i$. This is a contradiction since $w \in N(v)$ and the vertices $x$ and $y$ are not stacked. 
 
Now we may assume that the vertices $x$ and $y$ are stacked, and assume without loss of generality that $y$ lies on $x$-direction $i$. This means that $k=0$ and $l=2$. If $x$ is not a leaf in $F(T)$, the root vertex $z$ of $F(T)$ is such that $(0,z,p',q')$ is adjacent to $v$ for some vectors $p'$ and $q'$. But this means that $(0,z,p',q')$ is also adjacent to $u=(0,x,p,q)$ which is a contradiction.
If $x$ is a leaf, since we are assuming that $y$ lies on $x$-direction $i$, $x$  must be the root vertex and $u=(0,x,p,0)$.
Then $v=(2,y,r,0)$, and $w=(2,z,p',0)\in N(u)$ with $z$ the unique neighbour of $x$.  Since $z$ and $y$ are stacked, $w\notin N(v)$ which is a contradiction.
Hence $x=y$.
\end{proof}

\begin{coro}
If $T$ is a binary tree, then there are no two vertices having the same non empty neighbourhood in $K_4T$.
\end{coro}
\begin{proof}
Assume $v$ and $w$ have the same neighbourhood. By Proposition \ref{vec}, we must have that $v,w \in \{(0,x,1,2),(0,x,3,0),(2,x,1,0)\}$.
If $ v= (2,x,1,0)$ and $w \neq v$, then for every vertex $y$ such that $y$ lies on $x$-direction 0, we have that if $x$ lies on $y$-direction $i$, then $u \in N(v)$ where $u=\left\{ \begin{array}{cc}
 (2,y,1,0) & \mbox{ if } i=0 \\
  (0,y,3,0)   & \mbox{ if } i=1  \\
    (0,y,1,2)  & \mbox{ if } i=2. \\
\end{array}\right.$ But, since $w \neq v$, we have that $u$ can not be adjacent to $w$. This means that there are no vertices lying on $x$-direction 0, but this implies that $N(v) = \emptyset$. So assume that $ v= (0,x,1,2)$ and $ w=(0,x,3,0)$. Then again, if there exist vertices $y,z \in T $ such that $y$ lies on $x$-direction 1 and $z$ lies on $x$-direction two, the vertex $(2,y,1,0)$ is adjacent to $v$ but not  to $w$ and the vertex $(2,z,1,0)$ is adjacent to $w$ but not to $v$. This means that $x$ is a leaf in $F(T)$, but this implies that $N(v)= \emptyset = N(w)$.
\end{proof}
\section{Recovering the tree $T$}
We shall first consider the case when $T$ is a binary tree before considering the general case, since this case is easier both notation-wise and mathematically.
\subsection{Binary trees}
Throughout this section we shall assume that $T$ is a binary tree. The main reason for this is Proposition \ref{vec}.
If $T$ is a binary tree, then $T$ is completely determined by $F(T)$.
 \begin{prop}
Assume $T$ is a binary tree and let $K_m$ be a complete subgraph of $K_4T$ for $m\geq 3$. Then in $F(T)$ there exists an independent set $I$ of $m$ vertices. 
\end{prop}
\begin{proof}
Let $v_i=(k_i, x_i, p_i, q_i)$ for $1\leq i \leq m$ be the vertices of $K_m$ in $K_4T$. Assume that $x_j$ lies on $x_i$-direction $t\in \{1,2\}$  for some $i\neq j$. Then $v_j=(2,x_j,1,0)$ and $v_i=(0,x_i,3,0)$ or $v_i=(0,x_i,1,2)$ (if $t=1$ or $t=2$ respectively). Since $v_i$ multiplies every vertex $v_l$, this means that $x_l$ lies on $x_i$-direction $t$ for every $l \in\{ 1, \dots,m\} \setminus \{i\} $. 
Recall that $v_j$, and hence every $v_l$, must be of the form $(2, x_l,1,0)$ for $l \in\{ 1, \dots,m\} \setminus \{i\} $. This means that $x_l$ lies on $x_j$-direction $0$ and vice versa $x_j$ lies on $x_l$-direction $0$ for $j,l \in\{ 1, \dots,m\} \setminus \{i\} $. This immediately implies that $\{x_1, \dots,x_m\} \setminus \{x_i\} $ is an independent set of $m-1$ vertices. Now assume the edge $(x_i, x_j)$ exists in $F(T)$ for some $j \in \{ 1, \dots,m\} \setminus \{i\} $. Then $x_j$ lies on $x_i$-direction $0$ but this implies that $x_i$ also lies on $x_l$-direction  for $l \in\{ 1, \dots,m\} \setminus \{i,j\} $ which is a contradiction. Hence $I =\{x_1, \dots, x_m\}$ is an independent set in $F(T)$.
\end{proof}

\begin{coro}\label{kp}
Assume that $T$ is a binary tree, and that $F(T)$ has $m$ leaves. Then there exists in $K_4T$ a subgraph isomorphic to $K_m$ such that each vertex in $K_m$ is adjacent to exactly one leaf. 
\end{coro}
Let $KP_m$ be the subgraph obtained by a attaching a leaf to each vertex of a complete subgraph $K_m$ (see Figure \ref{n1} (right)).

Let $Kp$ denote the subgraph of $K_4T$ consisting of the complete subgraph mentioned in Corollary \ref{kp} together with the adjacent leaves.  If we delete from $K_4T$ the subgraph $Kp$, we obtain the subgraph of $K_4T$ consisting only of vertices $(k,x,p,q)$ such that $x$ is not a leaf of $F(T)$. Moreover, the amount of leaves in $K_4T-Kp$ corresponds to the amount of leaves in $F(T)-L(F(T))$ and we can apply Corollary \ref{kp} again.

This means that if $S_1$ be the subgraph of $K_4T$ isomorphic to $KP_{m_1}$ for some $m_1$, then $K_4T-S_1$ contains $S_2$, a subgraph isomorphic to $KP_{m_2}$. We can continue this way  until $K_4T-\cup_{i=1}^nS_i$ is isomorphic to either a $KP_3$, a $KP_2=P_4$ or a $KP_1=K_2$ (since $T$ is binary), and notice that $m_1 \geq m_2 \geq \dots \geq m_n$. 

\begin{defi}
Let $V(K_4T)= \cup_{i=1}^n S_i$ such that the induced subgraph generated by $S_i$ is isomorphic to $KP_{m_i}$ as in the previous discussion. 
    We say that a vertex $v \in V(K_4T)$ belongs to level $i$ if it is a vertex of $S_i$. Moreover, $V(S_i)=L_i \cup R_i$ where $L_i$ is the set of leaves of $S_i$.
\end{defi}

We are now ready to recover the tree $T$ from $K_4T$ when $T$ is a binary tree. First, notice that $V(K_4T)= \cup_{i=1}^nV(S_i)$ and since each $S_i$ has $2m_i$ vertices, we are going to label each vertex of $L_i$ and $R_i$ with the labels $\{x_1^i, \dots, x_{m_i}^i\}$ and $\{y_1^i, \dots, y_{m_i}^i\}$ respectively, thus in $S_i$, $y_j^i$ is the only neighbour of $x_j^i$. We are going to construct a tree $F_T$ isomorphic to $F(T).$

\begin{defi} The tree $F_T$ is defined as follows:

   \begin{itemize}
   
       \item The vertices of $F_T$ are $V(F_T)=z_0 \cup \left( \bigcup_{i=1}^n \bigcup_{j=1}^{m_i} z_j^i\right)$, and
       \item $N(z_0)= \bigcup_{j=1}^{m_n}z_j^n$
       \item $(z_j^i,z_l^k)\in E(F_T)$ if $k=\mbox{min}\{m>i:(x_l^m, y_j^i)\in E(K_4T)\}.$
   \end{itemize} 
\end{defi}
We now must prove that this tree is indeed isomorphic to $F(T)$.
\begin{lema}\label{iso}
The trees $F(T)$ and $F_T$ are isomorphic.
\end{lema}
\begin{proof}
We proceed by induction over $n$, and assume $n=1$. This means that $K_4T=S_1\cong KP_{m_1}$ which means that $F(T)$ is a star $F(T) \cong K_{1,m_1}$. On the other hand, $F_T$ consists only of the vertices $z_0$ together with $\cup_{i=1}^{m_1}z_i^1$, but since $n=1$ we have that every vertex is adjacent to $z_0$, thus $F_T= K_{1,m_1}$.

Assume now, that the lemma is valid for values smaller than $n$. We have that $K_4T\setminus S_1 = K_4T'$ where $T'$ is such that $F(T')$ is obtained from $F(T)$ by removing every leaf. Then $V(K_4T')= \cup_{i=2}^n S_i$ thus by induction hypothesis, $F_{T'}=F(T')=F(T) \setminus L(F(T))$. 
Notice that the amount of vertices in $F(T)\setminus F(T')$ is the amount of leaves in $F(T)$ which is $m_1$, which is also the amount of vertices in $F_T\setminus F_{T'}$.
Take a vertex $w \in F(T)$, and let $v \in F(T)$ be a  leaf adjacent to $w$. By Corollary \ref{coro}, one of the following holds:
\begin{itemize}
    \item the vertex $(0,v,1,2)$ is a leaf in $K_4T$ and $(2,w,1,0)$ is its unique neighbour,
    \item the vertex $(0,v,3,0)$ is a leaf in $K_4T$ and $(2,w,1,0)$ is its unique neighbour, or
    \item the vertex $w$ is the root vertex of $T$, the vertex $(2,v,1,0)$ is a leaf in $K_4T$, and its unique neighbour is $(0,w,3,0)$.
\end{itemize}
Assume the first case holds. Then  $(0,v,1,2)=x_j^1$ and $(2,w,1,0)=y_i^1$ for $1 \leq i,j \leq m_1.$
 Since $v$ is not a leaf in $F(T),$ there exists $u \in V(F(T))$ such that $(v,u)\in E(F(T))$ and $u$ lies on $v$-direction 0. This means that $(2,v,1,0)$ is adjacent to $(0,u,3,0)$ or $(0,u,1,2)$. We can assume without loss of generality that  $(0,u,3,0)$ is adjacent to $(2,v,1,0)$, and thus it must also be adjacent to $(2,w,1,0)=y_i^1.$ Since $u$ is a neighbour of $v$, $(0,u,3,0)$  is a leaf in $K_4T$ and hence is a leaf in $S_2$, thus $(0,u,3,0)=x_l^2$. Then $(x_l^2, y_i^1) \in E(K_4T)$ and thus $(z_i^1, z_l^2) \in E(F_T)$.  This means that given an edge $(w,v) \in E(F(T))$ outside of $F(T')$, there exists an edge $(z_i^1, z_l^2) \in E(F_T)$. The other two cases are analogous.

 Now consider an edge $(z_i^1, z_j^2) \in F_T$ outside of $F_{T'}.$ This means that $(x_j^2, y_i^1) \in E(K_4T)$ with $x_j^2\in L_2$ and $ y_i^1 \in R_1$. Since $y_i^1 \in R_1$, again by Corollary \ref{coro} we have that $y_i^1 =(2,w,1,0)$ or $y_i^1 =(2,w,1,0)$ if $w$ is the root vertex of $F(T)$. Assume that $y_i^1 =(2,w,1,0)$, then $x_i^1 =(0,v,3,0) $ or $(0,v,1,2)$ and $v$ is the unique neighbour of the leaf $w$, thus $(v,w)\in F(T).$ If $y_i^1 =(2,w,1,0)$ with $w$ the root of $F(T)$, we have that $x_i^1 = (0,v,3,0)$ where $v$ is the unique neighbour of $w$. Hence for every edge in $F_T$ outside of $F_{T'}$ there exists an edge in $F(T)$ outside of $F(T')$. Thus $F(T)= F_T.$
\end{proof}

\subsection{The general case}
Let $RK_4T$ be the graph obtained from $K_4T$ as follows. For every set of vertices $\{v_1, \dots, v_{m}\}$ such that $N(v_i)=N(v_j)$ for $  i \neq j$, we remove from $K_4T$ the vertices $\{v_2, \dots, v_m\}$. Thus in $RK_4T$ there are no two vertices having the same neighbourhood.
Notice that the vertices of the graph $RK_4T$ can also be separated into $m_k$ sets $S_1, \dots, S_k$ such that each $S_i \cong KP_{m_i}$ for some $m_i.$ Assume again that $\{x_1^i \dots, x_{m_i}^i\}$ and $\{y_1^i \dots, y_{m_i}^i\}$ denote the set of leaves  and non-leaf vertices of $S_i$ for $1 \leq i \leq k$ respectively.  We then have the following definition.
\begin{defi} Assume $V(RK_4T)=\cup_{i=1}^k V(S_i)$ with $|V(S_i)|=m_i.$ The tree $RF_T$ is defined as follows.
  \begin{itemize}
       \item The vertices of $RF_T$ are $V(F_T)=z_0\cup\left( \cup_{i=1}^n \cup_{j=1}^{m_i} z_j^i\right)$ , and
       \item $N(z_0)= \cup_{j=1}^{m_n}z_j^n$,
       \item $(z_j^i,z_l^k)\in E(RF_T)$ if $k=\mbox{min}\{m>i:(x_l^m, y_j^i)\in E(RK_4T)\}.$
   \end{itemize} 
\end{defi}
The proof of the following Corollary is analogous to the proof of Lemma \ref{iso}.
\begin{coro}
The trees $RF_T$ and $F(T)$ are isomorphic.
\end{coro}
Since $T$ is not binary it is not fully determined by $F(T).$ Consider the set of vertices $LR=\{x_1, \dots, x_{m_1}\}$ of $K_4T$, which have degree one in $RK_4T$. Assume that $x_i$ has degree $d_i $ in $K_4T$ for $1 \leq i \leq m_1$. Then $d_i = \frac{1}{2}(k_i-1)(k_i-2)$ for some integer $k_i$.

\begin{prop}
    If a leaf $x \in RK_4T$ has degree $d=\frac{1}{2}(k-1)(k-2)$ in $K_4T$ for some $k$. Then the corresponding leaf $z \in F(T)$ has degree $k$ in $T$.
\end{prop}
    \begin{proof}
        Let $z$ be a leaf in $F(T)$, and let $w$ be its unique neighbour in $F(T)$. Assume also that $z$ lies on $w$ direction $i$. If $ z \neq \star$, then $i \neq 0$ and consider a vertex $v=(0,w, p,q) \in K_4T$ having $2$ or $3$ as the $i$th entry of $p$ or as the $i-l(p)$th entry of $q$. They all have the same neighbourhood consisting of vertices of the form $(2,z,p,0)$, and by Proposition \ref{num}, there are $d=\frac{1}{2}(k-1)(k-2)$ such vertices, where $d(z)=k$. Thus all the vertices of the form of $v$ correspond to a leaf $x \in RK_4T$ which has degree $d=\frac{1}{2}(k-1)(k-2)$. If $z=\star $, we consider vertices of the form $v=(0,z, p,q)$  having $2$ or $3$ as the $i$th entry of $p$ or as the $i-l(p)$th entry of $q$, and they all have the same neighbourhood consisting of vertices of the form $(2, w,p,0)$. Thus these last vertices correspond to a leaf $x\in RK_4T$ which  has the desired degree in $K_4T.$
        
    \end{proof}
This means that we have fully recovered $T$ from the graph $K_4T$.

\subsection*{Declarations}
On behalf of all authors, the corresponding author states that there is no conflict of interest. This manuscript has no associated data.

\end{document}